\documentclass[oneside]{amsart}
\usepackage[a4paper]{geometry}
\usepackage{amsmath}
\usepackage{amsthm}
\usepackage{amsfonts}
\usepackage{graphics}

\usepackage{amssymb}
\usepackage{verbatim}
\usepackage{amscd}
\usepackage{enumerate}
\usepackage{graphicx}
\usepackage{siunitx}
\usepackage{color}
\numberwithin{equation}{section}

\usepackage{mathabx}
\usepackage{mathtools}
\usepackage[tableposition=top]{caption}
\usepackage{booktabs,dcolumn}

\theoremstyle{plain}

\newtheorem{theorem}{Theorem}[section]
\newtheorem{proposition}[theorem]{Proposition}

\newtheorem{lemma}[theorem]{Lemma}
\newtheorem{corollary}[theorem]{Corollary}

\theoremstyle{definition}

\newtheorem{remark}[theorem]{Remark}

\usepackage{lipsum}

\usepackage{appendix}

\usepackage[numbers]{natbib}

\usepackage{hyperref}

\usepackage{color}

\allowdisplaybreaks

\begin{document}

\title {A note on $L^2$-boundary integrals of the Bergman kernel}

\keywords{Bergman kernel; Pluricomplex Green function; Weighted Bergman projections; $L^2$-boundary integrals}

\email{ttp754@uowmail.edu.au}
  
\author[Phung Trong Thuc]{Phung Trong Thuc}

\address
	{Institute for Mathematics and its Applications, School of Mathematics and Applied Statistics, University of Wollongong, Wollongong, NSW 2522, AUSTRALIA.}	
\thanks{The author was supported by PhD scholarship in ARC grant DE140101366.}	

\subjclass[2010]{Primary 32A25; Secondary 32W05.}
	
\begin{abstract}
For any bounded convex domain $\Omega$ with $C^{2}$ boundary in $\mathbb{C}^{n}$,
we show that there exist positive constants $C_{1}$ and $C_{2}$ such that
\[
C_{1}\sqrt{\dfrac{K\left(w,w\right)}{\delta\left(w\right)}}\leq\left\Vert K\left(\cdot,w\right)\right\Vert _{L^{2}\left(\partial\Omega\right)}\leq C_{2}\sqrt{\dfrac{K\left(w,w\right)}{\delta\left(w\right)}},
\]
for any $w\in\Omega$. Here $K$ is the Bergman kernel of $\Omega$, and $\delta$ is the distance-to-boundary
function. 
\end{abstract}

\maketitle
\section{Introduction}

Throughout this note, $\Omega$ will be a bounded pseudoconvex domain
in $\mathbb{C}^{n}$, and $\delta:\Omega\rightarrow\mathbb{R}^{+}$
is the distance function to the boundary, $\delta\left(z\right):=\inf\left\{ \left\Vert z-w\right\Vert :w\in\partial\Omega\right\} $. The Bergman kernel $K\left(z,w\right)$ of $\Omega$ is the reproducing kernel for
the space of square-integrable holomorphic functions. That is, $K:\Omega\times\Omega\rightarrow\mathbb{C}$ is the function
such that 
\[
f\left(z\right)=\intop_{\Omega}K\left(z,w\right)f\left(w\right)dV\left(w\right),\forall f\in A^{2}\left(\Omega\right).
\]
Here $A^{2}\left(\Omega\right)$ denotes the space of square-integrable
holomorphic functions on $\Omega$, and $dV$ is Lebesgue measure.

It was suggested by Catlin (\cite{Catl}) for studying the boundary behavior of the Bergman kernel, by starting with simple domains, such as ellipsoids. We are interested in obtaining estimates on the $L^{2}$-boundary
norm of the Bergman kernel for convex domains. When $\Omega$ is a smoothly convex domain,  the function
$K\left(\cdot,w\right)$ is smooth up to the boundary, for any fixed
$w\in\Omega$, due to the work of Boas and Straube, see \cite{BoSt91}. Thus the boundary integral $\intop_{\partial\Omega}\left|K\left(z,w\right)\right|^{2}d\sigma\left(z\right)$
is well-defined. Here $d\sigma$ is the standard surface measure. Even if we drop the smoothness of $\partial\Omega$, we can still
conclude that $K\left(\cdot,w\right)\in L^{2}\left(\partial\Omega\right)$
in the trace sense. This follows from the fact that any bounded convex
domain has a Lipschitz boundary (\cite{Gri11}), and the Bergman projection is bounded
in the Sobolev space $W^{\left.1\right/2}\left(\Omega\right)$, see \cite{MiSh01}. To avoid technical reasons, we will restrict ourselves to the case $\partial\Omega\in C^{2}$. The purpose of this paper is to estimate the norm $\left\Vert K\left(\cdot,w\right)\right\Vert _{L^{2}\left(\partial\Omega\right)}$
as $w$ varies in $\Omega$.

The main result is stated as follows.

\begin{theorem}\label{Main-1}
Let $\Omega$ be a bounded convex domain with $C^{2}$ boundary in $\mathbb{C}^{n}$. Then
\begin{equation}\label{eq:t4}
C_{1}\sqrt{\dfrac{K\left(w,w\right)}{\delta\left(w\right)}}\leq\left\Vert K\left(\cdot,w\right)\right\Vert _{L^{2}\left(\partial\Omega\right)}\leq C_{2}\sqrt{\dfrac{K\left(w,w\right)}{\delta\left(w\right)}},
\end{equation}
for any $w\in\Omega$. Here $C_{1}$ is a positive constant depending on $\Omega$,
and $C_{2}=\sqrt{4en+1}$.
\end{theorem}

Our method is elementary, basically being
based on a weighted version of Bergman projections by Berndtsson
and Charpentier (\cite{BeCh00}), and relations between the pluricomplex
Green function and the Bergman kernel. Our approach is motivated by the work of Chen and Fu (\cite{ChFu11}) on
the comparison of the Bergman and Szeg{\"o}  kernels. In fact, by the definition of the Szeg{\"o}  kernel, from Theorem
\ref{Main-1} we obtain that for any bounded convex domain $\Omega$
in $\mathbb{C}^{n}$,
\[
\dfrac{S\left(z,z\right)}{K\left(z,z\right)}\geq\dfrac{\delta\left(z\right)}{4en+1},\forall z\in\Omega.
\]
Here $S$ is the Szeg{\"o}  kernel of $\Omega$. We note that in \cite{ChFu11}, the authors obtained estimates on
$\left.S\right/K$ for a much larger class of domains called $\delta$-regular,
which includes pseudoconvex domains of finite type and pseudoconvex
domains having plurisubharmonic defining functions at boundary points. Our argument here can be extended to domains admitting a plurisubharmonic defining function, see Remark \ref{gener}. It seems to us that our approach and the stated result (Theorem \ref{Main-1}) have not been noticed in the literature.

In Section \ref{wei}, we provide a weighted version of Bergman projections. We recall some properties of the pluricomplex Green function in Section \ref{plugreen}. The proof of Theorem \ref{Main-1} is given in the last section.


\textbf{Notation.} 
We will use the notation $X\lesssim Y$ (resp. $X\gtrsim Y$) to denote the estimate $\left|X\right|\leq CY$ (resp. $X\geq C\left|Y\right|$),
for some positive constant $C$.  We use $X\approx Y$ for the
fact $X\lesssim Y\lesssim X$.

\textbf{Acknowledgements.} The author thanks his supervisors: Jiakun Liu and Tran Vu Khanh for helpful suggestions and encouragement. The author would like to thank Professor Bo-Yong Chen for his comments. 

\section{Weighted Bergman projections}\label{wei}

Let $\psi$ be a Lebesgue measurable function on $\Omega$. By $L^{2}\left(\Omega,e^{\psi}\right)$
we denote the Hilbert space of measurable functions associated with
the norm 
\[
\left\Vert f\right\Vert _{L^{2}\left(\Omega,e^{\psi}\right)}:=\sqrt{\intop_{\Omega}\left|f\right|^{2}e^{\psi}dV}.
\]

Let $P$ be the Bergman projection associated to $\Omega$. It can
be represented as 
\begin{equation}
P\left(f\right)\left(z\right):=\intop_{\Omega}K\left(z,w\right)f\left(w\right)dV\left(w\right),\forall f\in L^{2}\left(\Omega\right).\label{eq:deBer}
\end{equation}
The formula \eqref{eq:deBer} allows us to extend the domain of definition
of $P$. Let $f$ be a Lebesgue measurable function on $\Omega$,
we say that $P\left(f\right)$ is well-defined on $\Omega$ if for
almost every $z\in\Omega$, we have $K\left(z,\cdot\right)f\left(\cdot\right)\in L^{1}\left(\Omega\right)$. For example, when $\Omega$ is a smoothly bounded pseudoconvex domain
of finite type, $P\left(f\right)$ is well-defined for any $f\in L^{p}\left(\Omega\right)$,
with $p\geq1$ (see, e.g. \cite{Boa_k87,Bel86}).

The main purpose of this section is to establish the following result.

\begin{proposition}\label{Pro-1}
Let $\Omega$ be a bounded pseudoconvex domain. Let $0<r<1$ and let $\psi$ be a locally bounded, plurisubharmonic function on $\Omega$ such that $ri\partial\overline{\partial}\psi\geq i\partial\psi\wedge\overline{\partial}\psi$
as currents.
Then for any Lebesgue measurable function $f$ so that $P\left(f\right)$
is well-defined, the following inequality holds
\begin{equation}\label{maeq1}
\displaystyle \intop_{\Omega}\left|P\left(f\right)\left(z\right)\right|^{2}e^{\psi\left(z\right)}dV\left(z\right)\leq \dfrac{1}{1-r}\intop_{\Omega}\left|f\left(z\right)\right|^{2}e^{\psi\left(z\right)}dV\left(z\right).
\end{equation}
\end{proposition}

\begin{remark}
The inequality \eqref{maeq1} was already stated in \cite{BeCh00} with an implicit constant in the right-hand side. It turns out that the constant $\left.1\right/\left(1-r\right)$ here allows us to establish the estimates in Theorem \ref{Main-1}. 

We also recall that the condition $r\imath\partial\overline{\partial}\psi\geq\imath\partial\psi\wedge\overline{\partial}\psi$ is equivalent to the statement that $-e^{-\left.\psi\right/r}$ is plurisubharmonic on $\Omega$. 
\end{remark}

\begin{proof} It suffices to consider the case that the right hand side of \eqref{maeq1} is finite. We will employ a similar approach as in \cite{BeCh00}. The idea is to shift the standard $L^{2}$ space to weighted one
by Kohn's formula and by being more careful with the use of weighted
Bergman projections.

We first assume that $\psi\in C^{2}\left(\overline{\Omega}\right)$ and $f\in L^{2}\left(\Omega\right)$. By Kohn's formula (see e.g. \cite{Koh63,ChSh01})
\begin{equation}
P\left(f\right)=e^{-\psi}P_{\psi}\left(e^{\psi}f\right)-\overline{\partial}^{\star}N\left(\overline{\partial}\left(e^{-\psi}P_{\psi}\left(e^{\psi}f\right)\right)\right),\label{eq:3}
\end{equation}
where $P_{\psi}$ denotes the Bergman projection in $L^{2}\left(\Omega,e^{-\psi}\right)=L^{2}\left(\Omega\right)$.
Set $g:=e^{-\psi}P_{\psi}\left(e^{\psi}f\right)$ and $u:=\overline{\partial}^{\star}N\left(\overline{\partial}\left(e^{-\psi}P_{\psi}\left(e^{\psi}f\right)\right)\right)$. Since $P_{\psi}\left(e^{\psi}f\right)\in A^{2}\left(\Omega\right)$, 
we have $\intop_{\Omega}u\overline{g}e^{\psi}=0$. Thus
\begin{equation}
\intop_{\Omega}\left|P\left(f\right)\right|^{2}e^{\psi}=\intop_{\Omega}\left|g-u\right|^{2}e^{\psi}=\intop_{\Omega}\left|g\right|^{2}e^{\psi}+\intop_{\Omega}\left|u\right|^{2}e^{\psi}.\label{eq:peq1}
\end{equation}
For the first term of \eqref{eq:peq1}, we get 
\begin{equation}
\intop_{\Omega}\left|g\right|^{2}e^{\psi}=\intop_{\Omega}\left|P_{\psi}\left(e^{\psi}f\right)\right|^{2}e^{-\psi}\leq\intop_{\Omega}\left|e^{\psi}f\right|^{2}e^{-\psi}=\intop_{\Omega}\left|f\right|^{2}e^{\psi}.\label{eq:1}
\end{equation}
Note that $u=\overline{\partial}^{\star}N\left(-g\wedge\overline{\partial}\psi\right)$, and
$ue^{\psi}$ belongs to the orthogonal complement of $\ker\overline{\partial}$
in $L^{2}\left(\Omega,e^{-\psi}\right)$. Thus we obtain
\begin{equation}\label{eq:peq6}
\intop_{\Omega}\left|ue^{\psi}\right|^{2}e^{-\psi}\leq\intop_{\Omega}\left|\overline{\partial}\left(ue^{\psi}\right)\right|_{\imath\partial\overline{\partial}\psi}^{2}e^{-\psi}.
\end{equation}
It continues 
\begin{eqnarray}
\intop_{\Omega}\left|u\right|^{2}e^{\psi} & \leq & \intop_{\Omega}\left|\overline{\partial}u+\overline{\partial}\psi\wedge u\right|_{\imath\partial\overline{\partial}\psi}^{2}e^{\psi}\nonumber\\
 & = & \intop_{\Omega}\left|-g\wedge\overline{\partial}\psi+\overline{\partial}\psi\wedge u\right|_{\imath\partial\overline{\partial}\psi}^{2}e^{\psi}\label{eq:peq4} \\
 & \leq & r\intop_{\Omega}\left(\left|u\right|^{2}+\left|g\right|^{2}\right)e^{\psi},\nonumber 
\end{eqnarray}
here the last inequality follows by $\intop_{\Omega}u\overline{g}e^{\psi}=0$
and $\left|\overline{\partial}\psi\right|^{2}_{\imath\partial\overline{\partial}\psi}\leq r$. Therefore 
\begin{equation}
\intop_{\Omega}\left|u\right|^{2}e^{\psi}\leq\dfrac{r}{1-r}\intop_{\Omega}\left|g\right|^{2}e^{\psi}.\label{eq:2}
\end{equation}
From \eqref{eq:peq1}, \eqref{eq:1} and \eqref{eq:2}, we get the estimate \eqref{maeq1}.

When $f\in L^{2}\left(\Omega\right)$, but $\psi$ is just a locally bounded, plurisubharmonic on $\Omega$. Consider a sequence
of pseudoconvex domains $\left\{ \Omega_{j}\right\} $ such that $\overline{\Omega}_{j}\Subset\Omega_{j+1}$ 
and $\Omega=\bigcup_{j=1}^{\infty}\Omega_{j}$. For $\varepsilon>0$, denote the convolution $\psi_{\varepsilon}:=\psi\star\eta_{\varepsilon}$
the standard regularization. For each $j$, we can choose $\varepsilon_{j}$
such that $0<\varepsilon_{j}<\text{dist}\left(\Omega_{j},\partial\Omega\right)$
and 
\begin{equation}
\intop_{\Omega_{j}}\left|f\right|^{2}\left(e^{\psi_{\varepsilon_{j}}}-e^{\psi}\right)\leq\dfrac{1}{j}.\label{eq:pne1}
\end{equation}
This is due to the monotone convergence theorem and the fact $f\in L^{2}\left(\Omega\right)\cap L^{2}\left(\Omega,e^{\psi}\right)$.
Since $ri\partial\overline{\partial}\psi_{\varepsilon_{j}}\geq i\partial\psi_{\varepsilon_{j}}\wedge\overline{\partial}\psi_{\varepsilon_{j}}$,
by applying the previous argument, we get that
\[
\intop_{\Omega_{j}}\left|P_{\Omega_{j}}\left(f\right)\right|^{2}e^{\psi_{\varepsilon_{j}}}\leq\dfrac{1}{1-r}\intop_{\Omega_{j}}\left|f\right|^{2}e^{\psi_{\varepsilon_{j}}}.
\]
By $\psi_{\varepsilon_{j}}\geq\psi$ on $\Omega_{j}$ and \eqref{eq:pne1},
it continues 
\begin{equation}
\intop_{\Omega_{j}}\left|P_{\Omega_{j}}\left(f\right)\right|^{2}e^{\psi}\leq\dfrac{1}{1-r}\intop_{\Omega_{j}}\left|f\right|^{2}e^{\psi_{\varepsilon_{j}}}\leq\dfrac{1}{1-r}\left(\dfrac{1}{j}+\intop_{\Omega}\left|f\right|^{2}e^{\psi}\right).\label{eq:pne2}
\end{equation}
We conclude that for each fixed $k\in\mathbb{Z}^{+}$, the sequence $\left\{ P_{\Omega_{j}}\left(f\right)\right\} _{j=k}^{\infty}$
is uniformly bounded in $L^{2}\left(\Omega_{k}\right)$. By Cantor's
diagonal argument, we can assume, by passing to a subsequence, that $ P_{\Omega_{j}}\left(f\right)$ converges
weakly to a function $v$ in $L^{2}\left(\Omega,\text{loc}\right)$.
It is clear that $P_{\Omega_{j}}\left(f\right)e^{\left.\psi\right/2}$
also converges weakly to $ve^{\left.\psi\right/2}$ in $L^{2}\left(\Omega,\text{loc}\right)$.
Thus for each $K\Subset\Omega$, by \eqref{eq:pne2}
\[
\intop_{K}\left|v\right|^{2}e^{\psi}\leq\liminf_{j\rightarrow\infty}\intop_{K}\left|P_{\Omega_{j}}\left(f\right)\right|^{2}e^{\psi}\leq\dfrac{1}{1-r}\intop_{\Omega}\left|f\right|^{2}e^{\psi}.
\]
It follows that 
\[
\intop_{\Omega}\left|v\right|^{2}e^{\psi}\leq\dfrac{1}{1-r}\intop_{\Omega}\left|f\right|^{2}e^{\psi}.
\]
We now prove that $v=P_{\Omega}\left(f\right)$. First, since $\overline{\partial}\left(P_{\Omega_{j}}\left(f\right)\right)=0$
and $P_{\Omega_{j}}\left(f\right)\rightarrow v$ weakly, $v$ is holomorphic
in $\Omega$. It remains to show that 
\begin{equation}
\intop_{\Omega}\left|f-v\right|^{2}\leq\intop_{\Omega}\left|f-h\right|^{2},\forall h\in A^{2}\left(\Omega\right).\label{eq:peq2}
\end{equation}
To see this, fix any $K\Subset\Omega$, we have 
\begin{eqnarray*}
\intop_{K}\left|f-v\right|^{2} & \leq & \liminf_{j\rightarrow\infty}\intop_{K}\left|f-P_{\Omega_{j}}\left(f\right)\right|^{2}\\
 & \leq & \liminf_{j\rightarrow\infty}\intop_{\Omega_{j}}\left|f-P_{\Omega_{j}}\left(f\right)\right|^{2}\\
 & \leq & \liminf_{j\rightarrow\infty}\intop_{\Omega_{j}}\left|f-h\right|^{2}\\
 & \leq & \intop_{\Omega}\left|f-h\right|^{2}.
\end{eqnarray*}
So \eqref{eq:peq2} follows.

Finally, we consider the case when we only require that $\intop_{\Omega}\left|f\right|^{2}e^{\psi}$
is finite and $P\left(f\right)$ is well-defined. Set $f_{k}:=\chi_{\Omega_{k}}f$, where the sequence $\left\{ \Omega_{k}\right\} $
is the same as above, and $\chi_{\Omega_{k}}$ is the indicator function of $\Omega_{k}$. We have $f_{k}\in L^{2}\left(\Omega\right)\cap L^{2}\left(\Omega,e^{\psi}\right)$
and 
\[
\intop_{\Omega}\left|f_{k}-f\right|^{2}e^{\psi}=\intop_{\left.\Omega\right\backslash \Omega_{k}}\left|f\right|^{2}e^{\psi}\rightarrow0\;\text{as }k\rightarrow\infty.
\]
By the previous estimate, 
\[
\intop_{\Omega}\left|P\left(f_{k}\right)\right|^{2}e^{\psi}\leq\dfrac{1}{1-r}\intop_{\Omega}\left|f_{k}\right|^{2}e^{\psi}.
\]
It follows that $\left\{ P\left(f_{k}\right)\right\} $ is a Cauchy
sequence in $L^{2}\left(\Omega,e^{\psi}\right)$ and so converges to
a function $v$ in $L^{2}\left(\Omega,e^{\psi}\right)$. Thus 
\[
\intop_{\Omega}\left|v\right|^{2}e^{\psi}\leq\dfrac{1}{1-r}\intop_{\Omega}\left|f\right|^{2}e^{\psi}.
\]
Now, we only need to prove that $v=P\left(f\right)$. Since we can
choose a subsequence of the $\left\{ P\left(f_{k}\right)\right\} $
that converges pointwise to $v$, it suffices to show that $P\left(f_{k}\right)$
converges pointwise to $P\left(f\right)$. For each $z\in\Omega$,
\begin{eqnarray*}
\left|P\left(f_{k}\right)\left(z\right)-P\left(f\right)\left(z\right)\right| & = & \left|\intop_{\Omega}K\left(z,w\right)f_{k}\left(w\right)dw-\intop_{\Omega}K\left(z,w\right)f\left(w\right)dw\right|\\
 & = & \left|\intop_{\left.\Omega\right\backslash \Omega_{k}}K\left(z,w\right)f\left(w\right)dw\right|.
\end{eqnarray*}
The last integral goes to zero as $k\rightarrow\infty$ since $K\left(z,\cdot\right)f\left(\cdot\right)\in L^{1}\left(\Omega\right)$.

\end{proof}

\begin{remark}
The constant $\left.1\right/\left(1-r\right)$ is not sharp for any $r\in\left(0,1\right)$. More specifically, it is not sharp for $r\approx0$. 
To see this, by a result of B{\l}ocki (\cite[p. 89]{Blo04}),  for any function $v\perp\ker\overline{\partial}$ in $L^{2}\left(\Omega,e^{-\psi}\right)$, we have
\begin{equation}
\intop_{\Omega}\left|v\right|^{2}e^{-\psi}\leq\dfrac{4r}{\left(1-r\right)^{2}}\intop_{\Omega}\left|\overline{\partial}v\right|_{i\partial\overline{\partial}\psi}^{2}e^{-\psi}.\label{eq:peq7}
\end{equation}
Now, apply \eqref{eq:peq7} with $v:=\overline{\partial}^{\star}N\left(-g\wedge\overline{\partial}\psi\right)e^{\psi}$,
then we can replace the inequality \eqref{eq:peq6} by 
\[
\intop_{\Omega}\left|ue^{\psi}\right|^{2}e^{-\psi}\leq\dfrac{4r}{\left(1-r\right)^{2}}\intop_{\Omega}\left|\overline{\partial}\left(ue^{\psi}\right)\right|_{i\partial\overline{\partial}\psi}^{2}e^{-\psi}.
\]
Continue the argument there, we get that 
\[
\intop_{\Omega}\left|P\left(f\right)\right|^{2}e^{\psi}\leq\dfrac{1}{1-\frac{4r^{2}}{\left(1-r\right)^{2}}}\intop_{\Omega}\left|f\right|^{2}e^{\psi},
\]
provided $0<r<\frac{1}{3}$. Note that the constant $\left.1\right/\left(1-\left.4r^{2}\right/\left(1-r\right)^{2}\right)$
is sharper than $\left.1\right/\left(1-r\right)$.

Nevertheless, if we call $C\left(r\right)$ the sharp constant for
the estimate \eqref{maeq1}, that is, given $0<r<1$, $C\left(r\right)$ is the
least constant such that for any pseudoconvex domain $\Omega$ and
$ri\partial\overline{\partial}\psi\geq i\partial\psi\wedge\overline{\partial}\psi$,
we have
\[
\intop_{\Omega}\left|P\left(f\right)\right|^{2}e^{\psi}\leq C\left(r\right)\intop_{\Omega}\left|f\right|^{2}e^{\psi},
\]
then we can show that 
\begin{equation}
\lim_{r\rightarrow1}\dfrac{C\left(r\right)}{\frac{1}{1-r}}=1.\label{eq:pegg1}
\end{equation}
Therefore, the constant $\left.1\right/\left(1-r\right)$ is sharp
in the use of $r\rightarrow1$. 
To see \eqref{eq:pegg1}, choose $\Omega=\mathbb{D}$ the unit disc
in $\mathbb{C},$ $f=\left(-\log\left|z\right|\right)^{r}$ and $\psi=-r\log\left(-\log\left|z\right|\right)$.
We can easily check that 
\[
\dfrac{\intop_{\Omega}\left|P\left(f\right)\right|^{2}e^{\psi}}{\intop_{\Omega}\left|f\right|^{2}e^{\psi}}=\dfrac{\pi r}{\sin\left(\pi r\right)}.
\]
Thus 
\[
\dfrac{\pi r}{\sin\left(\pi r\right)}\leq C\left(r\right)\leq\dfrac{1}{1-r},
\]
and \eqref{eq:pegg1} follows.
\end{remark}

\begin{remark}
By a duality argument, under the same hypothesis as in Proposition
\ref{Pro-1}, we also have 
\[
\intop_{\Omega}\left|P\left(f\right)\left(z\right)\right|^{2}e^{-\psi\left(z\right)}dV\left(z\right)\leq\dfrac{1}{1-r}\intop_{\Omega}\left|f\left(z\right)\right|^{2}e^{-\psi\left(z\right)}dV\left(z\right).
\]
\end{remark}

Proposition \ref{Pro-1}, together with an idea of Chen (\cite{Che17}),
gives the following result.

\begin{corollary}\label{colLp}
Let $\Omega$ be a bounded pseudoconvex domain with a positive Diederich-Fornaess
index $\eta$. Then for any $0<t<\eta$ and $1\leq q<\left.4n\right/\left(2n-t\right)$, the Bergman projection $P$ associated to $\Omega$ is bounded
from $L^{2}\left(\Omega,\delta^{-t}\right)$ to $L^{q}\left(\Omega\right)$.
\end{corollary}

\begin{proof}
Since $\left.4n\right/\left(2n-t\right)>2$, it suffices to assume that $q>2$. Recall that the Diederich-Fornaess index of $\Omega$ is defined by
\[
\eta\left(\Omega\right):=\sup\left\{ \alpha\in\left[0,1\right]:\exists h\in PSH\left(\Omega\right)\text{ and }C>0\text{ such that }\dfrac{1}{C}\delta^{\alpha}<-h<C\delta^{\alpha}\right\} .
\]
Thus, we can choose $t'\in\left(t,\eta\right)$ and $h\in PSH\left(\Omega\right)$
such that $\frac{1}{C}\delta^{t'}<-h<C\delta^{t'}$, for some positive
constant $C$. Set $\psi:=-\left(\left.t\right/t'\right)\log\left(-h\right)$, 
then $\psi\in L_{\text{loc}}^{\infty}\left(\Omega\right)$ and 
\[
\frac{t}{t'}i\partial\overline{\partial}\psi\geq i\partial\psi\wedge\overline{\partial}\psi.
\]
For any measurable function $f$ such that $\intop_{\Omega}\left|f\right|^{2}\delta^{-t}$
is finite (which also implies $f\in L^{2}\left(\Omega\right)$), by
applying Proposition \ref{Pro-1}, we get that
\begin{equation}
\intop_{\Omega}\left|P\left(f\right)\right|^{2}\delta^{-t}\lesssim\intop_{\Omega}\left|P\left(f\right)\right|^{2}e^{\psi}\lesssim\intop_{\Omega}\left|f\right|^{2}e^{\psi}\lesssim\intop_{\Omega}\left|f\right|^{2}\delta^{-t}.\label{eq:cpq1}
\end{equation}
Thus, for any $\varepsilon>0,$ 
\begin{eqnarray*}
\intop_{\delta\leq\varepsilon}\left|P\left(f\right)\right|^{2} & \leq & \varepsilon^{t}\intop_{\delta\leq\varepsilon}\left|P\left(f\right)\right|^{2}\delta^{-t}\leq\varepsilon^{t}\intop_{\Omega}\left|P\left(f\right)\right|^{2}\delta^{-t}\\
 & \lesssim & \varepsilon^{t}\intop_{\Omega}\left|f\right|^{2}\delta^{-t}.
\end{eqnarray*}
Moreover, by the mean value inequality, 
\begin{eqnarray*}
\left|P\left(f\right)\left(z\right)\right|^{2} & \lesssim & \delta^{-2n}\left(z\right)\intop_{B\left(z,\delta\left(z\right)\right)}\left|P\left(f\right)\right|^{2}\\
 & \lesssim & \delta^{-2n}\left(z\right)\intop_{\delta\leq2\delta\left(z\right)}\left|P\left(f\right)\right|^{2}\\
 & \lesssim & \delta^{-2n+t}\left(z\right)\intop_{\Omega}\left|f\right|^{2}\delta^{-t}.
\end{eqnarray*}
It follows that for each $k\in\mathbb{Z}^{+},$ 
\begin{eqnarray*}
\intop_{2^{-k-1}<\delta\leq2^{-k}}\left|P\left(f\right)\right|^{q} & \lesssim & 2^{-\left(k+1\right)\left(q-2\right)\left(-n+\frac{t}{2}\right)}\left(\intop_{\Omega}\left|f\right|^{2}\delta^{-t}\right)^{\frac{q}{2}-1}\intop_{\delta\leq2^{-k}}\left|P\left(f\right)\right|^{2}\\
 & \lesssim & 2^{-k\left(t-\left(q-2\right)\left(n-\frac{t}{2}\right)\right)}\left(\intop_{\Omega}\left|f\right|^{2}\delta^{-t}\right)^{\frac{q}{2}}.
\end{eqnarray*}
Similarly, 
\begin{eqnarray*}
\intop_{\delta>\left.1\right/2}\left|P\left(f\right)\right|^{q} & \lesssim & \left(\frac{1}{2}\right)^{\left(q-2\right)\left(-n+\frac{t}{2}\right)}\left(\intop_{\Omega}\left|f\right|^{2}\delta^{-t}\right)^{\frac{q}{2}-1}\intop_{\Omega}\left|P\left(f\right)\right|^{2}\\
 & \lesssim & \left(\intop_{\Omega}\left|f\right|^{2}\delta^{-t}\right)^{\frac{q}{2}-1}\intop_{\Omega}\left|f\right|^{2}\\
 & \lesssim & \left(\intop_{\Omega}\left|f\right|^{2}\delta^{-t}\right)^{\frac{q}{2}}.
\end{eqnarray*}
Therefore
\begin{eqnarray*}
\intop_{\Omega}\left|P\left(f\right)\right|^{q} & \leq & \intop_{\delta>\left.1\right/2}\left|P\left(f\right)\right|^{q}+\sum_{k=1}^{\infty}\intop_{2^{-k-1}<\delta\leq2^{-k}}\left|P\left(f\right)\right|^{q}\\
 & \lesssim & \left(1+\sum_{k=1}^{\infty}2^{-k\left(t-\left(q-2\right)\left(n-\frac{t}{2}\right)\right)}\right)\left(\intop_{\Omega}\left|f\right|^{2}\delta^{-t}\right)^{\frac{q}{2}}\\
 & \lesssim & \left(\intop_{\Omega}\left|f\right|^{2}\delta^{-t}\right)^{\frac{q}{2}},
\end{eqnarray*}
where the last inequality follows by the hypothesis $q<\left.4n\right/\left(2n-t\right)$.
\end{proof}

Note that we do not impose any regularity assumption on the boundary
of $\Omega$ in Corollary \ref{colLp}. In the case when $\partial\Omega$
is Lipschitz, it is known that $\eta\left(\Omega\right)$ is positive,
see \cite{Har08}. Moreover, using the Lipschitz property, we can conclude
that $\delta^{-\alpha}\in L^{1}\left(\Omega\right)$ for any $\alpha<1$
(see e.g. \cite{Gri11}).
Thus, using H{\"o}lder's inequality, we get that if $p\in\left(2,\infty\right)$
and $\left.tp\right/\left(p-2\right)<1$ then 
\[
\intop_{\Omega}\left|f\right|^{2}\delta^{-t}\leq\left(\intop_{\Omega}\left|f\right|^{p}\right)^{\frac{2}{p}}\left(\intop_{\Omega}\delta^{\frac{-tp}{p-2}}\right)^{\frac{p-2}{p}}\lesssim\left(\intop_{\Omega}\left|f\right|^{p}\right)^{\frac{2}{p}}.
\]
Combining this with Corollary \ref{colLp}, we conclude that for a
given $q\in\left[2,\left.4n\right/\left(2n-\eta\right)\right)$, if
there exists $t$ such that 
\[
\eta>t>2n-\dfrac{4n}{q}\text{ and }1-\dfrac{2}{p}>t,
\]
then the Bergman projection $P$ maps from $L^{p}\left(\Omega\right)$
to $L^{q}\left(\Omega\right)$ continuously. This requirement on $t$
is equivalent to $p>\left.2q\right/\left(q+2n\left(2-q\right)\right)$. 

We therefore arrive at the following result:

\begin{corollary}\label{colLp1}
Let $\Omega$ be a bounded pseudoconvex domain with Lipschitz boundary.
Let $\eta$ be the Diederich-Fornaess index of $\Omega$. Then for
any $q\in\left[2,\left.4n\right/\left(2n-\eta\right)\right)$, the
Bergman projection associated to $\Omega$ is bounded from $L^{p}\left(\Omega\right)$
to $L^{q}\left(\Omega\right)$, provided that $p>\left.2q\right/\left(q+2n\left(2-q\right)\right)$.
\end{corollary}

\begin{remark}
Corollary \ref{colLp1} says that for domains with Lipschitz boundary, one can always gain the regularity of the output space to an exponent bigger than $2$ $\left(\text{i.e. }L^{q},q>2\right)$,
given that the regularity of the input is high enough. This is not the case for non-Lipschitz domains. For instance, consider the Hartogs triangle domain $\Omega_{\gamma}:=\left\{ \left(z_{1},z_{2}\right)\in\mathbb{C}^{2}:\left|z_{1}\right|^{\gamma}<\left|z_{2}\right|<1\right\} $, with $\gamma>0$ and $\gamma\notin\mathbb{Q}$. This is a non-Lipschitz pseudoconvex domain (see e.g.
\cite{ZWO99}).  It is known that for any $p\in\left(1,\infty\right)$ and $q>2$, the Bergman projection associated
to $\Omega_{\gamma}$ cannot be bounded from $L^{p}\left(\Omega_{\gamma}\right)$
to $L^{q}\left(\Omega_{\gamma}\right)$, see \cite[p. 2681]{EdMcN17}.

\end{remark}

\section{The pluricomplex Green function}\label{plugreen}

We now recall some well-known results of the pluricomplex Green function. Let $\Omega$ be a bounded pseudoconvex domain in $\mathbb{C}^{n}.$
The pluricomplex Green function with a pole $w\in\Omega$ is defined
by 
\[
G\left(\cdot,w\right):=\sup\left\{ u\left(\cdot\right):u\in PSH^{-}\left(\Omega\right),\limsup_{z\rightarrow w}\left(u\left(z\right)-\log\left|z-w\right|\right)<\infty\right\} .
\]
Here $PSH^{-}\left(\Omega\right)$ denotes the set of all negative plurisubharmonic functions on $\Omega$. The following results are used in the sequel.

\begin{proposition}[Herbort \cite{Her99}, B{\l}ocki \cite{Blo15}]\label{HER}
Let $\Omega$ be a bounded pseudoconvex domain and let $t$ be any positive
number. Then
\begin{enumerate}
\item For any $f\in A^{2}\left(\Omega\right)$ and any $w\in\Omega,$ 
\begin{equation}\label{herbo1}
\intop_{\left\{ G\left(\cdot,w\right)<-t\right\} }\left|f\left(z\right)\right|^{2}dz\geq e^{-2nt}\dfrac{\left|f\left(w\right)\right|^{2}}{K\left(w,w\right)}.
\end{equation}

\item For any $w\in\Omega,$ 
\begin{equation}\label{herbo2}
K\left(w,w\right)\geq e^{-2nt}K_{\left\{ G\left(\cdot,w\right)<-t\right\} }\left(w,w\right),
\end{equation}
where $K_{\left\{ G\left(\cdot,w\right)<-t\right\} }$ denotes the
Bergman kernel of $\left\{ G\left(\cdot,w\right)<-t\right\} $.
\end{enumerate}

\end{proposition}

\begin{proposition}[Chen \cite{ChFu11}, B{\l}ocki \cite{Blo05}]\label{BLO}
Let $\Omega$ be a bounded pseudoconvex domain in $\mathbb{C}^{n}.$
Assume that there exists a plurisubharmonic function $\varphi$ on
$\Omega$ such that for any $z\in\Omega$, 
\[
C_{1}\delta^{a}\left(z\right)\leq-\varphi\left(z\right)\leq C_{2}\delta^{b}\left(z\right),
\]
where $C_{1},C_{2},a$ and $b$ are positive constants. Then there
exist positive constants $C$ and $\delta_{0}$ such that 
\begin{equation}\label{eq:bneww2}
\left\{ G\left(\cdot,w\right)<-1\right\} \subset\left\{ \dfrac{1}{C}\delta^{\left.a\right/b}\left(w\right)\left|\log\delta\left(w\right)\right|^{\left.-1\right/b}\leq\delta\left(\cdot\right)\leq C\delta^{\left.b\right/a}\left(w\right)\left|\log\delta\left(w\right)\right|^{\left.n\right/a}\right\} ,
\end{equation}
for any $w\in\Omega$ and $\delta\left(w\right)<\delta_{0}$.

If  $\Omega$ is a convex domain then 
\begin{equation}
\left\{ G\left(\cdot,w\right)<-t\right\} \subset\left\{ \dfrac{e^{t}-1}{e^{t}+1}\delta\left(w\right)\leq\delta\left(\cdot\right)\leq\dfrac{e^{t}+1}{e^{t}-1}\delta\left(w\right)\right\} ,\label{eq:bnew1}
\end{equation}
for any $w\in\Omega$ and any $t >0$.
\end{proposition}

For several applications of the pluricomplex Green function, we refer readers to  \cite{ChFu11, Blo05, Che16, Che17, Blo13}. The following result can be obtained by using these interesting properties.

\begin{proposition}\label{ProAp1}
Let $\Omega\subset\mathbb{C}^{n}$ be a strongly pseudoconvex domain with smooth boundary
and let $\alpha\in\mathbb{R}$. If the Bergman-Toeplitz operator $T_{\alpha}$, defined by
\[
f\rightarrow T_{\alpha}\left(f\right)\left(z\right):=\intop_{\Omega}K\left(z,w\right)f\left(w\right)\delta^{\alpha}\left(w\right)dV\left(w\right),
\]
is bounded from $L^{p}\left(\Omega\right)$ to $L^{q}\left(\Omega\right)$,
with $1<p\leq q<\infty$ then 
\[
\alpha\geq\left(n+1\right)\left(\dfrac{1}{p}-\dfrac{1}{q}\right).
\]
\end{proposition}

\begin{remark}
This result has been obtained in \cite{ARS12} by using several estimates of Kobayashi balls and $\theta$-Carleson measures in Bergman spaces. The proof given below is a direct consequence of Proposition \ref{HER} and Proposition \ref{BLO}. Note that the converse statement is also true, i.e. if $\alpha\geq\left(n+1\right)\left(\left(\left.1\right/p\right)-\left(\left.1\right/q\right)\right)$
then $T_{\alpha}$ maps from $L^{p}\left(\Omega\right)$ to $L^{q}\left(\Omega\right)$
continuously, see \cite{CuMcN06,KLT17}.
\end{remark}

\begin{proof}
Without loss of generality we may assume that $\alpha\geq0$. First, it is well-known that for any strongly pseudoconvex domain $\Omega$ with smooth boundary, there is a positive constant $C\left(\Omega\right)$ such that
\begin{equation}
K\left(z,z\right)\geq C\delta^{-n-1}\left(z\right),\label{eq:p1}
\end{equation}
for any $z\in\Omega$. Moreover, for any $p>1$, there is a constant $C\left(p,\Omega\right)$ such that
\begin{equation}\label{eq:p2}
\left\Vert K\left(z,\cdot\right)\right\Vert _{L^{p}\left(\Omega\right)}\leq C\delta^{-\left(n+1\right)\left(1-\frac{1}{p}\right)}\left(z\right),
\end{equation}
for any $z\in\Omega$, see e.g. \cite[Theorem 2.7]{ARS12}, also \cite{Li92,CuMcN06}. 

Now, using properties of the Bergman projection, we have
\begin{eqnarray*}
\intop_{\Omega}\left|K\left(z,w\right)\right|^{2}\delta^{\alpha}\left(w\right)dw & = & \intop_{\Omega}K\left(w,z\right)\delta^{\alpha}\left(w\right)K\left(z,w\right)dw\\
 & = & \intop_{\Omega}K\left(w,z\right)\delta^{\alpha}\left(w\right)\left(\intop_{\Omega}K\left(z,\xi\right)K\left(\xi,w\right)d\xi\right)dw\\
 & = & \intop_{\Omega}\left(\intop_{\Omega}K\left(\xi,w\right)\delta^{\alpha}\left(w\right)K\left(w,z\right)dw\right)K\left(z,\xi\right)d\xi.
\end{eqnarray*}
By  H{\"o}lder's inequality, it follows
\begin{eqnarray}
\intop_{\Omega}\left|K\left(z,w\right)\right|^{2}\delta^{\alpha}\left(w\right)dw & \leq & \left\Vert \intop_{\Omega}K\left(\cdot,w\right)\delta^{\alpha}\left(w\right)K\left(w,z\right)dw\right\Vert _{L^{q}\left(\Omega\right)}\left\Vert K\left(z,\cdot\right)\right\Vert _{L^{q'}\left(\Omega\right)}\nonumber \\ 
 & \lesssim& \left\Vert K\left(z,\cdot\right)\right\Vert _{L^{p}\left(\Omega\right)}\left\Vert K\left(z,\cdot\right)\right\Vert _{L^{q'}\left(\Omega\right)}\label{eq:p3}\\
 & \lesssim& \delta^{-\left(n+1\right)\left(1-\frac{1}{p}+\frac{1}{q}\right)}\left(z\right).\nonumber 
\end{eqnarray}
Here, the second inequality comes from the boundedness of $T_{\alpha}$,
the third follows from \eqref{eq:p2}, and $q'$ is the dual exponent of $q$, i.e. $\left.1\right/q+\left.1\right/q'=1$. On the other hand, by using Proposition \ref{HER}, Proposition \ref{BLO} and \eqref{eq:p1},
we obtain
\begin{eqnarray*}
\intop_{\Omega}\left|K\left(z,w\right)\right|^{2}\delta^{\alpha}\left(w\right)dw & \geq & \intop_{\left\{ G\left(\cdot,z\right)<-1\right\} }\left|K\left(z,w\right)\right|^{2}\delta^{\alpha}\left(w\right)dw\\
 & \gtrsim & \delta^{\alpha}\left(z\right)\left|\log\delta\left(z\right)\right|^{-\alpha}\intop_{\left\{ G\left(\cdot,z\right)<-1\right\} }\left|K\left(z,w\right)\right|^{2}dw\\
 & \gtrsim & \delta^{\alpha}\left(z\right)\left|\log\delta\left(z\right)\right|^{-\alpha}K\left(z,z\right)\\
 & \gtrsim & \delta^{\alpha-n-1}\left(z\right)\left|\log\delta\left(z\right)\right|^{-\alpha}.
\end{eqnarray*}
From this and \eqref{eq:p3}, the conclusion follows by letting $z\rightarrow\partial\Omega$.
\end{proof}

\begin{remark}
A similar approach has been used in \cite{KLT18}  for Hartogs triangle domains.
\end{remark}

\section{Proof of Theorem \ref{Main-1}}\label{Main}

Let us first recall some facts from the theory of Hardy spaces. We refer readers to the books by Stein \cite{Ste72} and Krantz \cite{Kra01} for details.

Let $\Omega\subset\mathbb{C}^{n}$ be a bounded domain with $C^{2}$
boundary. The Hardy space $h^{2}\left(\Omega\right)$ is defined by
\[
h^{2}\left(\Omega\right):=\left\{ f\text{ harmonic on }\Omega:\intop_{\partial\Omega}\left|f\left(z\right)\right|^{2}d\sigma\left(z\right):=\limsup_{\varepsilon\rightarrow0^{+}}\intop_{\delta=\varepsilon}\left|f\left(z\right)\right|^{2}d\sigma\left(z\right)<\infty\right\} .
\]
There always exists a positive constant $\varepsilon_{0}$ depending on $\Omega$ such that
the following norms are equivalent
\[
\left(\intop_{\partial\Omega}\left|f\left(z\right)\right|^{2}d\sigma\left(z\right)\right)^{\frac{1}{2}}\text{ and }\left(\sup_{0<\varepsilon<\varepsilon_{0}}\intop_{\delta=\varepsilon}\left|f\left(z\right)\right|^{2}d\sigma\left(z\right)\right)^{\frac{1}{2}},
\]
for $f\in h^{2}\left(\Omega\right)$. Moreover, there is a constant $C$ depending on $\Omega$ such that 
\begin{equation}
\intop_{\Omega}\left|f\left(z\right)\right|^{2}dV\left(z\right)\leq C\intop_{\partial\Omega}\left|f\left(z\right)\right|^{2}d\sigma\left(z\right),\label{eq:hard1}
\end{equation}
for any $f\in h^{2}\left(\Omega\right)$. That is, the $L^{2}$-norm is dominated by the Hardy space norm for
functions in $h^{2}\left(\Omega\right)$. Finally, we will need the following result, see \cite[Lemma 2.2]{ChFu11}.

\begin{lemma}\label{chhard}
Let $\Omega\subset\mathbb{C}^{n}$ be a bounded domain with $C^{2}$
boundary. For any harmonic function $u$ on $\Omega$,
\begin{equation}
\limsup_{\varepsilon\rightarrow0^{+}}\intop_{\delta=\varepsilon}\left|u\left(z\right)\right|^{2}d\sigma\left(z\right)=\limsup_{r\rightarrow1^{-}}\left(1-r\right)\intop_{\Omega}\left|u\left(z\right)\right|^{2}\delta^{-r}\left(z\right)dV\left(z\right).\label{eq:hard2}
\end{equation}
\end{lemma}

We are ready to proceed Theorem \ref{Main-1}.

\begin{proof}[Proof of Theorem \ref{Main-1}]
Let $\varepsilon_{0}$ and $c_{1}$ be positive constants such that
\[
\sup_{0<\varepsilon<\varepsilon_{0}}\intop_{\delta=\varepsilon}\left|f\left(z\right)\right|^{2}d\sigma\left(z\right)\leq c_{1}\intop_{\partial\Omega}\left|f\left(z\right)\right|^{2}d\sigma\left(z\right),\forall f\in h^{2}\left(\Omega\right).
\]
We first assume that $c_{0}\delta\left(w\right)<\varepsilon_{0}$,
with $c_{0}:=\left.\left(e+1\right)\right/\left(e-1\right)$. By applying \eqref{eq:bnew1} in Proposition
\ref{BLO}, 
\begin{eqnarray}
\intop_{\left\{ G\left(\cdot,w\right)<-1\right\} }\left|K\left(z,w\right)\right|^{2}dV\left(z\right) & \leq & \intop_{\left\{ \delta\left(\cdot\right)\leq c_{0}\delta\left(w\right)\right\} }\left|K\left(z,w\right)\right|^{2}dV\left(z\right)\nonumber \\
 & \leq & \intop_{0}^{c_{0}\delta\left(w\right)}\left(\intop_{\delta=\varepsilon}\left|K\left(z,w\right)\right|^{2}d\sigma\left(z\right)\right)d\varepsilon\label{eq:fina1}\\
 & \leq & c_{0}c_{1}\delta\left(w\right)\intop_{\partial\Omega}\left|K\left(z,w\right)\right|^{2}d\sigma\left(z\right).\nonumber 
\end{eqnarray}
By using \eqref{herbo1} in Proposition \ref{HER}, 
\begin{equation}
\intop_{\left\{ G\left(\cdot,w\right)<-1\right\} }\left|K\left(z,w\right)\right|^{2}dV\left(z\right)\geq e^{-2n}K\left(w,w\right).\label{eq:fina2}
\end{equation}
Combining \eqref{eq:fina1} with \eqref{eq:fina2}, we conclude that 
\[
\left\Vert K\left(\cdot,w\right)\right\Vert _{L^{2}\left(\partial\Omega\right)}\geq C_{1}\sqrt{\dfrac{K\left(w,w\right)}{\delta\left(w\right)}},
\]
 for a positive constant $C_{1}$ depending on $\Omega$. 

For the case $c_{0}\delta\left(w\right)\geq\varepsilon_{0}$, using
\eqref{eq:hard1} we have 
\begin{eqnarray*}
\intop_{\partial\Omega}\left|K\left(z,w\right)\right|^{2}d\sigma\left(z\right) & \geq & C\intop_{\Omega}\left|K\left(z,w\right)\right|^{2}dV\left(z\right)\\
 & = & CK\left(w,w\right)\\
 & \geq & \dfrac{C\varepsilon_{0}}{c_{0}}\dfrac{K\left(w,w\right)}{\delta\left(w\right)}.
\end{eqnarray*}
Therefore we have proved the left-hand side of \eqref{eq:t4}.

We now turn to the proof of the right-hand side. Since $\Omega$ is
convex, the function $-\delta$ is convex on $\Omega$, and is thus
also plurisubharmonic on $\Omega$, see e.g. \cite{ArKu85}. By applying
Proposition \ref{Pro-1} with $\psi\left(z\right):=-r\log\left(\delta\left(z\right)\right)$,
we have 
\begin{equation}
\left(1-r\right)\intop_{\Omega}\left|P\left(f\right)\left(z\right)\right|^{2}\delta^{-r}\left(z\right)dV\left(z\right)\leq\intop_{\Omega}\left|f\left(z\right)\right|^{2}\delta^{-r}\left(z\right)dV\left(z\right),\label{eq:fina3}
\end{equation}
for any $0<r<1$ and any measurable function $f$. Inserting $$f\left(z\right):=\chi_{\left\{ G\left(\cdot,w\right)<-t\right\} }\left(z\right)K_{\left\{ G\left(\cdot,w\right)<-t\right\} }\left(z,w\right)$$
into \eqref{eq:fina3}, we obtain
\begin{equation}
\left(1-r\right)\intop_{\Omega}\left|K\left(z,w\right)\right|^{2}\delta^{-r}\left(z\right)dV\left(z\right)\leq\intop_{\left\{ G\left(\cdot,w\right)<-t\right\} }\left|K_{\left\{ G\left(\cdot,w\right)<-t\right\} }\left(z,w\right)\right|^{2}\delta^{-r}\left(z\right)dV\left(z\right),\label{eq:fina4}
\end{equation}
for any $t>0$. By Lemma \ref{chhard}, Proposition \ref{HER} and
Proposition \ref{BLO}, it continues 
\begin{eqnarray*}
\intop_{\partial\Omega}\left|K\left(z,w\right)\right|^{2}d\sigma\left(z\right) & = & \limsup_{r\rightarrow1^{-}}\left(1-r\right)\intop_{\Omega}\left|K\left(z,w\right)\right|^{2}\delta^{-r}\left(z\right)dV\left(z\right)\\
 & \leq & \intop_{\left\{ G\left(\cdot,w\right)<-t\right\} }\left|K_{\left\{ G\left(\cdot,w\right)<-t\right\} }\left(z,w\right)\right|^{2}\delta^{-1}\left(z\right)dV\left(z\right)\\
 & \leq & \dfrac{e^{t}+1}{e^{t}-1}\delta^{-1}\left(w\right)\intop_{\left\{ G\left(\cdot,w\right)<-t\right\} }\left|K_{\left\{ G\left(\cdot,w\right)<-t\right\} }\left(z,w\right)\right|^{2}dV\left(z\right)\\
 & = & \dfrac{e^{t}+1}{e^{t}-1}\delta^{-1}\left(w\right)K_{\left\{ G\left(\cdot,w\right)<-t\right\} }\left(w,w\right)\\
 & \leq & \dfrac{e^{t}+1}{e^{t}-1}e^{2nt}\dfrac{K\left(w,w\right)}{\delta\left(w\right)}.
\end{eqnarray*}
The desired inequality then follows by noting that $$\inf\left\{ \left.\left(e^{t}+1\right)e^{2nt}\right/\left(e^{t}-1\right):t>0\right\} <4en+1.$$
\end{proof}

\begin{remark}
Since the constant $C_{2}=\sqrt{4en+1}$ depends only on the dimension
$n$, it suggests for example a study of the sharp estimates in Theorem
\ref{Main-1}.
\end{remark}

\begin{remark}\label{gener}
The method used in the proof of Theorem \ref{Main-1} can be extended
to domains having a plurisubharmonic defining function, such as strongly
pseudoconvex domains and  Kohn special domains defined by 
\[
\Omega_{F}:=\left\{ z\in\mathbb{C}^{n}:\left|f_{1}\left(z\right)\right|^{2}+\ldots+\left|f_{m}\left(z\right)\right|^{2}<1\right\} ,
\]
where $F=\left(f_{1},\ldots,f_{m}\right):\mathbb{C}^{n}\rightarrow\mathbb{C}^{m}$
is a holomorphic map, see \cite{ChFu12}. It is known that for strongly pseudoconvex domains and Kohn special
domains, $\delta\left(z\right)\approx\delta\left(w\right)$ for $z\in\left\{ G\left(\cdot,w\right)<-1\right\} $,
see \cite{DiHe00,ChFu12,ChFu11}. As a result, the estimate $\left\Vert K\left(\cdot,w\right)\right\Vert _{L^{2}\left(\partial\Omega\right)}\thickapprox\left(\left.K\left(w,w\right)\right/\delta\left(w\right)\right)^{\left.1\right/2}$
holds true for these domains. For a general domain admitting a plurisubharmonic defining function,
we may use the estimates in \eqref{eq:bneww2}, which involve the
logarithmic terms. To be precise, let us state these as the following
corollary.
\end{remark}

\begin{corollary}
Let $\Omega$ be a bounded domain with $C^{2}$ boundary.
\begin{enumerate}
\item If $\Omega$ is either a strongly pseudoconvex domain or a Kohn special
domain then there exist positive constants $C_{1}$ and $C_{2}$ such
that for any $w\in\Omega$, 
\[
C_{1}\sqrt{\dfrac{K\left(w,w\right)}{\delta\left(w\right)}}\leq\left\Vert K\left(\cdot,w\right)\right\Vert _{L^{2}\left(\partial\Omega\right)}\leq C_{2}\sqrt{\dfrac{K\left(w,w\right)}{\delta\left(w\right)}}.
\]

\item If $\Omega$ is a pseudoconvex domain having a plurisubharmonic defining
function then there exist positive constants $C_{1}$ and $C_{2}$
such that for any $w\in\Omega$, 
\[
C_{1}\sqrt{\dfrac{K\left(w,w\right)}{\delta\left(w\right)\left|\log\delta\left(w\right)\right|^{n}}}\leq\left\Vert K\left(\cdot,w\right)\right\Vert _{L^{2}\left(\partial\Omega\right)}\leq C_{2}\sqrt{\dfrac{K\left(w,w\right)\left|\log\delta\left(w\right)\right|}{\delta\left(w\right)}}.
\]
\end{enumerate}
\end{corollary}

\end{document}